\documentclass{IEEEtran}
\usepackage{cite}
\usepackage{url}
\usepackage{amsmath,amssymb,amsfonts}
\usepackage{multirow}
\usepackage{amsthm}
\usepackage{xcolor}
\usepackage{graphicx}
\usepackage{textcomp}
\usepackage{epsfig}

\usepackage{epstopdf}
\usepackage[ruled]{algorithm}
\usepackage{algorithmic}
\usepackage{enumerate}
\usepackage{float}
\usepackage{graphics,graphicx}
\usepackage{subfigure}
\usepackage{bbm}

\usepackage{array}
\usepackage{mathtools}
\usepackage[justification=centering]{caption}
\allowdisplaybreaks[3]
\usepackage{url}
\DeclareGraphicsExtensions{.png}
\graphicspath{{./Pics/}}
\usepackage{color}
\usepackage{xspace}

\usepackage{hyperref}
\hypersetup{
    colorlinks=true,
    linkcolor=blue,
    filecolor=gray,      
    urlcolor=blue,
    citecolor=blue,
}

\def\BibTeX{{\rm B\kern-.05em{\sc i\kern-.025em b}\kern-.08em
		T\kern-.1667em\lower.7ex\hbox{E}\kern-.125emX}}

\newtheorem{theorem}{Theorem}

\newtheorem{assumption}{Assumption}
\newtheorem{proposition}{Proposition}

\begin{document}
	
	\title{\LARGE \bf An Accelerated Distributed Optimization with
Equality and Inequality Coupling Constraints}

\author{Chenyang Qiu, \IEEEmembership{Graduate Student Member, IEEE}, Yangyang Qian, \IEEEmembership{Member, IEEE}, Zongli Lin, \IEEEmembership{Fellow, IEEE}, and Yacov A. Shamash,
\IEEEmembership{Life Fellow, IEEE} 
\vspace{-1.5em}
\thanks{%
This work relates to Department of Navy Awards N00014-23-1-2124 and N00014-24-1-2287 issued by the Office of Naval Research. The United States
Government has a royalty-free license throughout the world in all
copyrightable material contained herein. \emph{(Corresponding author: Zongli Lin.)}} 
\thanks{%
Chenyang Qiu, Yangyang Qian, and Zongli Lin are with the Charles L. Brown Department of Electrical and Computer Engineering, University of Virginia, Charlottesville, VA 22904, USA (e-mail: nzp4an@virginia.edu; jbt4up@virginia.edu;  zl5y@virginia.edu).}
\thanks{%
Yacov A. Shamash is with the Department of Electrical and Computer
Engineering, Stony Brook University, Stony Brook, NY 11794, USA (e-mail:
yacov.shamash@stonybrook.edu).}} 
\maketitle	
 
\begin{abstract}
This paper studies distributed convex optimization with both affine equality and nonlinear inequality couplings through the duality analysis. We first formulate the dual of the coupling-constraint problem and reformulate it as a consensus optimization problem over a connected network.
To efficiently solve this dual problem and hence the primal problem, we design an accelerated linearized algorithm where, at each round, a look-ahead linearization of the separable objective is combined with a quadratic penalty on the Laplacian constraint, a proximal step, and an aggregation of iterations.
On the theory side, we prove non-ergodic rates for both the primal optimality error and the feasibility error. On the other hand, numerical experiments show a faster decrease of optimality error and feasibility residual than the state-of-the-art algorithms under the same communication budget.
\end{abstract}


\section{Introduction}\label{sec:introduction}

Distributed optimization seeks to minimize a global objective across a multi-agent system via localized computation and communication and finds many applications, including smart grids, sensor networks, and distributed learning \cite{yang2019survey}. A particularly challenging class of distributed optimization problems is the coupling-constraint problem (CCP), where agents must coordinate local decisions subject to global coupling constraints. Such couplings increase the complexity of the problem, especially in decentralized settings where no central coordinator exists and each agent has only local information.

A classical distributed optimization problem is the economic dispatch problem, also referred to as resource allocation \cite{yang2016distributed}. It is typically formulated as a CCP, where the goal is to minimize the total cost while meeting the demand and satisfying both the individual and the global limitations. Among these constraints, a global affine equality constraint commonly represents the power balance condition in smart grids, ensuring that total generation meets total demand \cite{wen2021recent}. Other physical limitations are modeled as coupling inequality constraints, such as emission constraints or network capacity limits, which capture interdependent physical relationships among different units \cite{carrillo2022effect,benidris2016emission}.

A lot of efforts have aimed to handle the coupled equality constraint, for instance, via the alternating direction method of multipliers (ADMM) \cite{chang2014multi,chen2017admm}, mirror-based approaches \cite{nedic2018improved}, and gradient-tracking of dual problem \cite{zhang2020distributed}. 
Besides, there are some works that focus on distributed optimization problems with coupling affine inequality constraints \cite{Li2019DistributedPA, Du2024DistributedAO, du2025linear}. However, these works are not suitable for optimization problems subject to nonlinear inequality constraints.

In particular, to address the challenging problems involving globally coupled nonlinear inequality constraints, formed by aggregating all local constraints across the nodes, various distributed algorithms have been developed \cite{liang2019distributed1,liang2019distributed2,liu2020unitary,wu2022distributed,falsone2023augmented, chang2014distributed}. 
In \cite{liang2019distributed1}, smooth convex programs with globally coupled inequalities are solved via a dual subgradient method with iterate-averaging feedback. The follow-up \cite{liang2019distributed2} introduces an operator-splitting primal–dual framework that handles both equality and inequality couplings. Ref. \cite{liu2020unitary} develops a distributed subgradient scheme combining dual methods with dynamic average consensus to iteratively minimize local dual models, and extends it via dual decomposition to handle functionally coupled constraints. In \cite{wu2022distributed}, the sparsely coupled and densely coupled constraints are efficiently handled using different techniques, respectively.
Ref. \cite{falsone2023augmented} introduces an Augmented Lagrangian Tracking (ALT) distributed algorithm to handle both affine equality and nonlinear inequality coupling constraints and provides an asymptotic convergence result under the assumption of convexity. 
Their convergence properties have also been established, ranging from asymptotic convergence \cite{chang2014distributed,falsone2023augmented, liang2019distributed1} to explicit rates such as ergodic $\mathcal{O}(\tfrac{\ln N}{\sqrt{N}})$ \cite{liang2019distributed2}, nonergodic $\mathcal{O}(1/\sqrt{N})$ \cite{liu2020unitary}, and ergodic $\mathcal{O}(1/N)$\cite{wu2022distributed}, where $N$ is the total number of iterations.

It is noted that a majority of existing distributed algorithms for solving the CCP do not consider accelerated convergence and therefore exhibit relatively slow ergodic convergence rates. 
Motivated by this limitation, this work aims to develop a distributed algorithm with a provable accelerated non-ergodic convergence rate, extending the theoretical guarantees of classical first-order methods to the CCP framework with general (possibly nonsmooth) cost functions. Nesterov introduced a fast first-order method for unconstrained convex optimization to enhance the efficiency of first-order methods, significantly accelerating the convergence rate to $O(1/N^2)$ \cite{nesterov1983method}. Later, numerous algorithms adapted Nesterov's acceleration idea for different optimization problems, such as the Fast Iterative Shrinkage-Thresholding Algorithm (FISTA) \cite{beck2009fast}, fast Lagrangian-based algorithms \cite{sabach2022faster,chen2025fast}, and distributed optimization algorithms \cite{cao2025dcatalyst,baranwal2024distributed}. Nevertheless, these algorithms do not apply to the distributed CCP. 

This paper proposes an algorithm for the distributed optimization problems with coupling equality and inequality constraints over an undirected communication network, significantly improving the nonergodic convergence rate, where the nonergodic rate reflects the convergence performance of individual iteration and does not rely on averaging. Specifically, we first transform the CCP into a dual problem. Given the separable nature of the objective function, this dual problem can be interpreted as a consensus optimization problem. By leveraging the linearized method of multipliers, we develop an accelerated distributed algorithm to solve the dual problem, thereby enabling an efficient distributed solution to the original CCP. In particular, we establish that at the $N$th iteration, the proposed algorithm achieves a non-ergodic convergence rate of $O(1/N^2)+O(1/N)$ for the primal optimality error, while the constraint violation decreases at a non-ergodic rate of $O(1/N^2)+O(1/N)$. These results improve upon existing prior works, where only ergodic or asymptotic convergence guarantees have been established, as in \cite{falsone2023augmented, chang2014distributed, liu2020unitary, liang2019distributed1, liang2019distributed2,wu2022distributed}.


The remainder of this paper is organized as follows. Section \ref{sec: Problem Formulation} introduces the problem formulation. Section \ref{sec: Algorithm Development} develops our accelerated distributed algorithm. Section \ref{sec: Convergence Analysis} carries out the convergence analysis. Section \ref{sec: Simulation} provides numerical experiments to validate our algorithm. Section \ref{sec: Conclusion} concludes the paper.

\textbf{Notation}: For a differentiable function $f: \mathbb{R}^{p} \rightarrow \mathbb{R}$, its gradient at $x \in \mathbb{R}^{p}$ is denoted by $\nabla f(x)$. 
For a non-differentiable function $f: \mathbb{R}^{p} \rightarrow \mathbb{R}$, $\partial f(x)$ represents a subgradient at $x \in \mathbb{R}^{p}$. 
A convex function \( f: X \rightarrow (-\infty, +\infty] \) is called proper if \( f(x) > -\infty \) for all \( x \in X \) and $f(x)$ is not trivially equal to $+\infty$. The relative interior of a set $S$, $\operatorname{ri}(S)$, is defined as $
\operatorname{ri}(S)=\left\{x \in S: \text { there exists } \epsilon>0 \text { such that } B_\epsilon(x) \cap \operatorname{aff}(S) \subseteq S\right\}
$, where $\operatorname{aff}(S)$ is the affine hull of $S$ and $B_\epsilon(x)$ is a ball of radius $\epsilon$ centered on $x$. Let $P_Y$ be the projection operator onto the convex set $Y$. 
We denote by $\mathbb{R}_{+}^p$ the nonnegative orthant in $\mathbb{R}^p$, i.e., $\mathbb{R}_{+}^p = \{ x \in \mathbb{R}^p \mid x_i \geq 0,\ \forall i = 1, \dots, p \}$. The projection of a vector $y \in \mathbb{R}^p$ onto $\mathbb{R}_{+}^p$ is denoted by $[y]_+$.
The symbols $0_{p}$, $1_{p}$, $O_{p}$, and $I_{p}$ are used to denote the $p$-dimensional all-zero vector, all-one vector, zero matrix, and identity matrix, respectively. Let $\otimes$ be the Kronecker product, $\langle\cdot, \cdot\rangle$ be the Euclidean inner product, and $\|\cdot\|$ be the $\ell_2$ norm. For a positive semidefinite matrix $A \succeq {O}_{p}$ and a vector ${x} \in \mathbb{R}^p,\|{x}\|_A^2= {x}^{\rm T} A  {x}$, $A^{\dagger}$ is the pseudoinverse of $A$, and $\lambda_2(A)$ is the smallest nonzero eigenvalue of matrix $A$. 

\section{Problem Formulation}\label{sec: Problem Formulation}
Consider the CCP for a distributed network of multiple agents. These agents can communicate with each other through a communication network. The communication network is modeled by a connected undirected graph $\mathcal{G}=(\mathcal{V}, \mathcal{E})$, where $\mathcal{V}=\{1,2, \ldots,n\}$ is the set of nodes and $\mathcal{E} \subseteq\{(i, j) \subseteq \mathcal{V} \times \mathcal{V} \mid i \neq j\}$ is the set of edges. For each agent $i \in \mathcal{V}$, the set of its neighbors is denoted by $\mathcal{N}_i=$ $\{j \in \mathcal{V} \mid (i, j) \in \mathcal{E}\}$. 
The objective of the CCP is to minimize the summation of the costs of all agents in the network while satisfying both the global coupling constraints and individual constraints. To be more specific, the formulation of CCP considered in this paper is defined as follows,
\begin{align}\label{economic dispatch problem}
 ~ \min _{x \in X}& \ f(x) = \sum_{i=1}^n f_i(x_i), \\
\operatorname{s.t.} & \ \sum_{i=1}^n B_i x_i = \sum_{i=1}^n b_i, \notag \\
 &\ \sum_{i=1}^n h_i(x_i) \leq 0 \notag,
\end{align}
where $x = [x_1^{\rm{T}} \;x_2^{\rm{T}} \ldots x_n^{\rm{T}}]^{\rm{T}} \in \mathbb{R}^{n p}$, $X_i \subseteq \mathbb{R}^p$ is the local constraint set for each $x_i$, and $X = X_1 \times X_2 \times \ldots \times X_n$. In the constraints, $B_i \in \mathbb{R}^{d \times p}$ and $b_i \in \mathbb{R}^d$, while $h_i$ is a possibly nonlinear function.

\begin{assumption}\label{ass: strongly convex}
For each $i \in \mathcal{V}$, we assume that
\begin{enumerate}
    \item $f_i(x_i): \mathbb{R}^{p} \rightarrow (-\infty, +\infty]$ is a proper and $\mu_f-$strongly convex function, and
    \item $h_i(x_i): \mathbb{R}^{p} \rightarrow \mathbb{R}^m$ is convex and $l_h-$Lipschitz continuous on the convex set $X_i$, i.e., there exists
    $l_h > 0$ such that,
    $$
    \! \! \! \! \| h_i(x_{i,2}) - h_i(x_{i,1}) \| \!\leq \! l_h \| x_{i,2} - x_{i,1} \|, \!~ \forall x_{i,2},\, x_{i,1} \in X_i.
    $$
\end{enumerate}
\end{assumption}


\begin{assumption}\label{ass: slater's condition}
    The set $X_i$ is compact and convex. Furthermore, the Slater condition is satisfied, that is, there exists at least one point $x$ in the relative interior $\operatorname{ri}\left(X\right)$ of $X$  such that both $\sum_{i=1}^n B_i x_i = \sum_{i=1}^n b_i$ and $\sum_{i=1}^n h_i(x_i) < 0$ are satisfied.
\end{assumption}

\section{Algorithm Development} \label{sec: Algorithm Development}
In this section, we first transform problem \eqref{economic dispatch problem} into a dual problem. Building upon classical primal–dual frameworks, we further embed an acceleration mechanism, combining extrapolation, adaptive scaling, and averaging steps, to achieve an accelerated rate.

\subsection{Dual Problem}
By introducing the Lagrangian multipliers $\mu \in \mathbb{R}^d$ and $\delta \in \mathbb{R}_{+}^m$, the Lagrangian function associated with problem \eqref{economic dispatch problem} is defined as 
\begin{align}
    L( x, \mu, \delta) = & \ \sum_{i=1}^n L_i(x_i, \mu, \delta) \notag \\
    = & \ \sum_{i=1}^n \left(F_i( x_i )  + \left\langle \mu,  B_i x_i - b_i \right \rangle + \left\langle \delta,  h_i(x_i) \right \rangle\right),\label{dualLagrangianfunc}
\end{align}
where $F_i \triangleq f_i+\mathbbm{1}_{X_i}: \mathbb{R}^p \rightarrow (-\infty, +\infty],~ i \in \mathcal{V}
$, $\mathbbm{1}_{X_i}(x_i)$ is the indicator function associated with the convex set $X_i$, $i \in \mathcal{V}$, i.e., $
\mathbbm{1}_{X_i}\left(x_i\right)= 0 \text { if } x_i \in X_i$ and $ \mathbbm{1}_{X_i}\left(x_i\right)= +\infty$, otherwise. The dual problem of problem \eqref{economic dispatch problem} is then given as
\begin{align*} 
    & \underset{\mu \in \mathbb{R}^d,\, \delta \in \mathbb{R}_+^m}{\max} \ \underset{ x }{\min } L( x, \mu, \delta) 
    = \ \underset{\mu \in \mathbb{R}^d,\, \delta \in \mathbb{R}_+^m}{\min} \sum_{i=1}^n g_i(\mu,\delta),
\end{align*}    
where for each $i \in \mathcal{V}$, $g_i(\mu,\delta)$ is defined as
\begin{align*}
     g_i(\mu, \delta) \coloneq -\underset{x_i}{\min} \left\{ L_i(x_i, \mu, \delta) \right\}.
\end{align*}
Keep a copy of the variables $\mu$ and $\delta$ at each agent $i \in \mathcal{V}$ as $\mu_i$ and $\delta_i$ to each node $i \in \mathcal{V}$ and then compact them into $y_i \coloneq [ \mu_i^{\rm{T}}~ \delta_i^{\rm{T}}]^{\rm{T}}  \in Y$, where $Y \coloneq \mathbb{R}^{d} \times \mathbb{R}_+^m$. To simplify notation, we do not distinguish between $g_i(\mu_i, \delta_i)$ and $g_i(y_i)$. We have the dual optimization problem as follows,
\begin{align}\label{consensusoptimizationproblem}
\min_{y_i \in Y} \ &\,\sum_{i=1}^n g_i(y_i),  \\  
\operatorname{s.t.} \ &\  y_1 = y_2 = \cdots = y_n. \notag
\end{align}
We denote the Laplacian matrix by $H$, where 
$$
[H]_{i j}=\left\{\begin{array}{ll}
\sum_{s \in \mathcal{N}_i} H_{i s}, & i=j, \\
-H_{i j}, & j \in \mathcal{N}_i, \\
0, & \text {otherwise, }
\end{array} \quad i, j \in \mathcal{V},\right.
$$
with $H_{i j}=H_{j i}>0$ being the weight of edge $\{i, j\} \in \mathcal{E}$. Since $\mathcal{G}$ is a connected undirected graph, the null space of $H$ is $\operatorname{span}\left\{ {1}_n\right\}$. Define $y\coloneq \left[y_1^{\rm{T}} \; y_2^{\rm{T}} \ldots y_n^{\rm{T}}\right]^{\rm{T}} \in \mathcal{Y}$, $\mathcal{Y} \coloneq Y^n$ and $W \coloneq H \otimes I_{d+m} \succeq {O}_{n({d+m})}$. Then, we obtain that the range spaces of the matrices $W$, $W^{\frac{1}{2}}$, $W^{\dagger}$ and $(W^{\dagger})^{\frac{1}{2}}$ are the same and equal to $\left\{y \in \mathbb{R}^{n (d+m)} \mid y_1+\cdots+y_n={0}_{d+m}\right\}$, which is
the orthogonal complement of 
$\left\{y \in \mathbb{R}^{n (d+m)} \mid y_1 =\cdots=y_n\right\}$. Hence, the consensus constraint $y_1=\cdots=y_n$ can be replaced with $W^{\frac{1}{2}} {y}={0}_{n(d+m)}$. In this way,  we reformulate problem \eqref{consensusoptimizationproblem} into the following compact form:
\begin{align}\label{compactdualoptimization}
    \min _{y \in \mathcal{Y} } ~ &\, G(y)=\sum_{i=1}^n g_i(y_i) \\
    \operatorname{s.t.} ~ &\, W^{\frac{1}{2}} y = {0}_{n(d+m)} \notag .
\end{align}

\begin{proposition}\label{prop:smoothness}
    Suppose Assumption~\ref{ass: strongly convex} holds. Then the function $g_i: Y \rightarrow \mathbb{R}$ is Lipschitz smooth with a constant $l_g$, i.e., for any $z_1, z_2 \in Y$,
    \begin{equation}\label{proposition1}
        \| \nabla g_i(z_1) - \nabla g_i(z_2) \| \leq l_g \| z_1 - z_2 \|,
    \end{equation}
    where $l_g = \sqrt{ \frac{2}{\mu_f^2} (\|B_i\|^2 + l_h^2) \cdot \max\{\|B_i\|^2, l_h^2\} }$.
\end{proposition}
\begin{proof}
    For any $i \in \mathcal{V}$, the strong convexity of $F_i$ and the convexity of $h_i$ imply that the Lagrangian function $L_i(x_i, \mu, \delta)$ is strongly convex in $x_i$. Hence, for any $z_1 = [\mu_1^{\rm{T}}~ \delta_1^{\rm{T}}]^{\rm{T}} \in Y$ and $z_2 = [\mu_2^{\rm{T}}~ \delta_2^{\rm{T}}]^{\rm{T}} \in Y$, the optimal solutions
    \begin{eqnarray*}
    x_{i,1} &:=& \underset{x_i}{\operatorname{argmin}} L_i(x_i, \mu_1, \delta_1), \\ 
    x_{i,2} &:=& \underset{x_i}{\operatorname{argmin}} L_i(x_i, \mu_2, \delta_2)
    \end{eqnarray*}
    are uniquely defined.
    
    By Danskin's Theorem \cite[Proposition B.22]{bertsekas1997nonlinear}, the gradient of $g_i$ with respect to $(\mu, \delta)$ is given by $\nabla g_i \coloneqq [ (\nabla_\mu g_i)^{\rm{T}} ~ (\nabla_\delta g_i)^{\rm{T}} ]^{\rm{T}} $, where
    $\nabla_\mu g_i(\mu_1,\delta_1) = -B_i x_{i,1} + b_i$, $\nabla_\mu g_i(\mu_2,\delta_2) = -B_i x_{i,2}+b_i$, and $\nabla_\delta g_i(\mu_1,\delta_1) = - h_i(x_{i,1})$, $\nabla_\delta g_i(\mu_2,\delta_2) = - h_i(x_{i,2})$. 
    
    By the definition of strong convexity, for any subgradient $\tilde{\nabla}_{x_i} L_i(x_{i,1}, \mu_1, \delta_1) \in \partial_{x_i} L_i(x_{i,1}, \mu_1, \delta_1)$, it holds that
    \begin{align*}
        & \ L_i(x_{i,2}, \mu_1, \delta_1) - L_i(x_{i,1}, \mu_1, \delta_1)\\
        \geq &\ \left\langle \tilde{\nabla}_{x_i} L_i(x_{i,1}, \mu_1, \delta_1), x_{i,2} - x_{i,1} \right\rangle + \frac{\mu_f}{2} \| x_{i,2} - x_{i,1} \|^2.
    \end{align*}
    Since $0 \in \partial_{x_i} L_i(x_{i,1}, \mu_1, \delta_1)$, the above simplifies to
    \begin{equation} \label{(2)-(1)}
        L_i(x_{i,2}, \mu_1, \delta_1) - L_i(x_{i,1}, \mu_1, \delta_1) \geq \frac{\mu_f}{2} \| x_{i,2} - x_{i,1} \|^2.
    \end{equation}
    Similarly, we obtain
    \begin{equation} \label{(1)-(2)}
        L_i(x_{i,1}, \mu_2, \delta_2) - L_i(x_{i,2}, \mu_2, \delta_2) \geq \frac{\mu_f}{2} \| x_{i,2} - x_{i,1} \|^2.
    \end{equation}
    Adding equations \eqref{(2)-(1)} and \eqref{(1)-(2)}, we obtain
    \begin{align*}
        & \mu_f \| x_{i,2} - x_{i,1} \|^2 \\
        \leq&\ \left( L_i(x_{i,2}, \mu_1, \delta_1) - L_i(x_{i,1}, \mu_1, \delta_1) \right) \\
        &\quad + \left( L_i(x_{i,1}, \mu_2, \delta_2) - L_i(x_{i,2}, \mu_2, \delta_2) \right) \\
        =& \ \langle \mu_1 - \mu_2, B_i(x_{i,2} - x_{i,1}) \rangle + \langle \delta_1 - \delta_2, h_i(x_{i,2}) - h_i(x_{i,1}) \rangle \\
        \leq& \ \left( \|B_i\| \| \mu_1 - \mu_2 \| + l_h \| \delta_1 - \delta_2 \| \right) \| x_{i,2} - x_{i,1} \|,
    \end{align*}
    where the last inequality uses the Lipschitz continuity of $h_i$. 
    Dividing both sides by $\| x_{i,2} - x_{i,1} \|$ (nonzero since otherwise the conclusion is trivial), we obtain
    \begin{equation}\label{continuityofx}
        \| x_{i,2} - x_{i,1} \|\leq \frac{\|B_i\|}{\mu_f } \| \mu_1-\mu_2 \| + \frac{l_h}{\mu_f } \| \delta_1 - \delta_2 \|. 
    \end{equation}
    
    With the additional assumption of Lipschitz continuity of $h_i$, we have
    $$
    \| \nabla_\delta g(\mu_1,\delta_1) - \nabla_\delta g(\mu_2,\delta_2)\| \leq l_h \| x_{i,1} - x_{i,2} \|,
    $$
    which, in view of \eqref{continuityofx}, results in
    \begin{align*}
        &\ \| \nabla g(\mu_1,\delta_1) - \nabla g(\mu_2,\delta_2)\|^2 \\
        =& \ \| \nabla_\mu g(\mu_1,\delta_1) - \nabla_\mu g(\mu_2,\delta_2)\|^2 \\
        & \ + \| \nabla_\delta g(\mu_1,\delta_1) - \nabla_\delta g(\mu_2,\delta_2)\|^2\\
        \leq &\ (\|B_i\|^2 + l_h^2) \| x_{i,1} - x_{i,2} \|^2 \\
        \leq &\ \frac{2}{\mu_f^2} (\|B_i\|^2 + l_h^2)\left( \|B_i\|^2 \| \mu_1-\mu_2 \|^2 + l_h^2 \| \delta_1 - \delta_2 \| ^2\right) \\
        \leq &\  l_g^2 ( \| \mu_1-\mu_2 \|^2 + \| \delta_1 - \delta_2 \| ^2 ).
    \end{align*}
    This establishes the Lipschitz smoothness of $g_i$ and completes the proof. 
\end{proof}

Under Assumptions \ref{ass: strongly convex} and \ref{ass: slater's condition}, the strong duality holds. The optimal solution of problem \eqref{economic dispatch problem}, denoted by $ x^* = [(x_1^*)^{\rm{T}}\;(x_2^*)^{\rm{T}} \ldots (x_n^*)^{\rm{T}}]^{\rm{T}}$, is unique.
In addition, we regard $ x^*$ and $y^*$ as the optimal pair if and only if
\begin{enumerate}
    \item $ x^*$ is feasible, i.e., $ \sum_{i=1}^n B_i x_i^* = \sum_{i=1}^n b_i$ and $\sum_{i=1}^n h(x_i^*) \leq 0$.
    \item $y^* = [(y_1^*)^{\rm{T}}\;(y_2^*)^{\rm{T}} \ldots (y_n^*)^{\rm{T}}]^{\rm{T}} \in \mathcal{Y}$ is an optimal solution to the dual problem \eqref{compactdualoptimization}, and $G(y^*) =- F( x^*)$.
\end{enumerate}


\subsection{Accelerated Distributed Algorithm}
In this subsection, we aim to develop a distributed algorithm for solving problem \eqref{compactdualoptimization} by introducing an accelerated linearized method of multipliers. To this end, we define the augmented Lagrangian function as follows,
$$L_\rho(y,  v)=G(y)- \left\langle v, W^{\frac{1}{2}} y\right \rangle +\frac{\rho}{2}\left\|W^{\frac{1}{2}} y\right\|^2, $$
where $ v = [v_1^{\rm{T}}\;v_2^{\rm{T}}\ldots v_n^{\rm{T}}]^{\rm{T}} \in \mathbb{R}^{n(d+m)}$ and $\rho>0$. Recall the method of multipliers \cite{boyd2011distributed}, where updates are given by
\begin{subequations}
\begin{alignat}{2}
    y_{k+1} & = \underset{y \in \mathcal{Y}}{\operatorname{argmin}}\left\{ G(y) \! - \!\left\langle  v_k, W^{\frac{1}{2}} y \right\rangle \!+ \!\frac{\rho}{2}\left\| W^{\frac{1}{2}} y \right\|^2 \right\}, \label{admm primal}\\
     v_{k+1} & = v_k-\rho W^{\frac{1}{2}} y_{k+1}, \label{admm dual}
\end{alignat}
\end{subequations}
where $k$ is the iteration index. 
Note that $G(y)$ is a function that contains the operator $\operatorname{argmin}$. 
In this case, we consider to linearize $G(y)$ using its first-order Taylor expression $G(y_k) + \langle \nabla G(y_k), y - y_k \rangle$. 
In addition, although $G(y)$ is separable, the linear function $\langle  v_k, W^{\frac{1}{2}} y \rangle$ and the quadratic term $\frac{\rho}{2}\| W^{\frac{1}{2}} y \|^2$ are not suitable for distributed computation. To deal with this issue, we define
\begin{equation}\label{lambda v}
    \lambda_k = W^{\frac{1}{2}}  v_k
\end{equation}
for distributed implementation. In view of \eqref{admm dual}, we have 
$$\lambda_{k+1} =\lambda_k-\rho W y_{k+1}.$$
With a proximal term $\frac{\eta}{2} \left\| y - y_k \right\|^2$, the linearized inexact updates of the method of multipliers are given by
\begin{subequations}
    \begin{alignat}{2}
        y_{k+1} =& \, \underset{y \in \mathcal{Y}}{\operatorname{argmin}} \Big\{\langle \nabla G(y_k), y \rangle - \langle \lambda_k, y \rangle + \rho \langle W y_k, y \rangle \notag \\
        & \, \left.+ \frac{\eta}{2} \left\| y - y_k \right\|^2\right\} \notag \\
        =&\, P_{\mathcal{Y}} \{ y_{k} - \frac{1}{\eta}(\nabla G(\tilde{y}_k) - \lambda_k - \rho W y_k)\}, \label{originialprimal}\\
        \lambda_{k+1} =& \,\lambda_k-\rho W y_{k+1},
    \end{alignat}
\end{subequations}
where the gradient of $G(y)$ is defined as
$$\nabla G(y) = \left[\nabla g_1^{\rm{T}}(y_1)~ \nabla g_2^{\rm{T}}(y_2)\, \ldots\, \nabla g_n^{\rm{T}}(y_n)\right]^{\rm{T}}.$$
We introduce the variables $\tilde{y}_k = [\tilde{y}_{1,k}^{\rm{T}} \; \tilde{y}_{2,k}^{\rm{T}} \ldots \tilde{y}_{n,k}^{\rm{T}} ]^{\rm{T}} \in \mathbb{R}^{n p}$ and $\hat{y}_{k} = [\hat{y}_{1, k}^{\rm{T}}\;\hat{y}_{2, k}^{\rm{T}} \ldots \hat{y}_{n,k}^{\rm{T}}]^{\rm{T}} \in \mathbb{R}^{n p}$. Then, we propose the following accelerated distributed algorithm for problem \eqref{compactdualoptimization}, \begin{subequations}\label{algorithms}
\begin{alignat}{5}
\tilde{y}_k= & \,\left(1-\alpha_k \right) \hat{y}_k+\alpha_k y_k, \label{ymd}\\
y_{k+1} = &\, P_{\mathcal{Y}} \{ y_{k} - \frac{1}{\eta_k}(\nabla G(\tilde{y}_k) - \lambda_k - \theta_k W y_k)\}, \label{gradient ykmd}\\
\hat{y}_{k+1}= & \,\left(1-\alpha_k\right) \hat{y}_k+\alpha_k y_{k+1}, \label{yag} \\
\lambda_{k+1}= & \,\lambda_k-\beta_k W y_{k+1}, \label{lambda update}
\end{alignat}    
\end{subequations}
where $\alpha_k, \eta_k, \theta_k, \beta_k \in \mathbb{R}$ are design parameters to be specified later, and $\nabla G(\tilde{y}_k) = [\nabla g_1^{\rm{T}}(\tilde{y}_{1,k})\; \nabla g_2^{\rm{T}}(\tilde{y}_{2,k}) \ldots$ $ \nabla g_n^{\rm{T}}(\tilde{y}_{n,k})]^{\rm{T}}$, $\nabla g_i (\tilde{y}_{i,k})$ in \eqref{gradient ykmd} can be obtained as follows,
\begin{equation}\label{nabla g}
    \nabla g_i (\tilde{y}_{i,k}) =-\begin{bmatrix}
    B_i x_{i,k} - b_i \\
    h_i(x_{i,k})
\end{bmatrix},
\end{equation}
where 
\begin{equation}\label{x_distributed_update}
    x_{i,k} = \underset{x}{\operatorname{argmin}}\left\{ F_i(x) +  \left \langle \begin{bmatrix}
    B_i x_{i} - b_i \\
    h_i(x_{i})
\end{bmatrix}, \tilde{y}_{i,k} \right\rangle \right\}.
\end{equation}
Specifically, the algorithm maintains a triplet that plays the roles of extrapolation and aggregation.
$\tilde{y}_k$ serves as an extrapolated prediction used for gradient evaluation, introducing a momentum-like effect that captures the trend of previous iterations, and $y_k$ performs the proximal correction ensuring stability. In addition, $\hat{y}_{k}$ smooths the trajectory and enables optimal convergence analysis. These variables can affect the trajectory of updates and provide a kind of foresight about where the updating direction is heading. By properly selecting the parameters, the algorithm achieves the accelerated rate typical of Nesterov-type acceleration methods \cite{nesterov1983method}.


In summary, the proposed accelerated distributed algorithm is detailed in Algorithm \ref{algorithm}, where $N$ denotes the number of iterations.   
{
\renewcommand{\baselinestretch}{1.05}
\begin{algorithm}[ht] 
\caption{Accelerated Distributed Algorithm for Economic Dispatch} 
\begin{algorithmic}[1]\label{algorithm}
    \STATE \textbf{Initialization:}
    \STATE For each node $i \in \mathcal{V}$, set $\hat{y}_{i,1} = y_{i,1} \in Y$ and $\lambda_{i,1} = 0$. 
    \FOR{$k = 1,2,\ldots, N $} 
        \STATE Each node $i \in \mathcal{V}$ sends the variable $y_{i,k}$ to its neighbors $j\in \mathcal{N}_i$.
        \STATE Each node $i \in \mathcal{V}$ updates the variable $\tilde{y}_{i,k}= \left(1-\alpha_k\right) \hat{y}_{i,k}+\alpha_k y_{i,k}$ and further updates the variable $x_{i,k}$ according to \eqref{x_distributed_update}.
        \STATE After receiving the information from its neighbors, each node $i \in \mathcal{V}$ computes the aggregated information $t_{i,k} = \sum_{j \in \mathcal{N}_i} H_{i j}(y_{i,k}-y_{j,k})$. 
        \STATE Each node $i \in \mathcal{V}$ updates $$y_{i,k+1} = P_{Y} \{y_{i,k} - \frac{1}{\eta_k}\left( \nabla g_i (\tilde{y}_{i,k}) - \lambda_{i,k} - \theta_k t_{i,k} \right) \},$$ where $\nabla g_i (\tilde{y}_{i,k})$ is defined in \eqref{nabla g}, and then updates $$\hat{y}_{i,k+1}= \left(1-\alpha_k\right) \hat{y}_{i,k}+\alpha_k y_{i,k+1}.$$
        \STATE Each node $i \in \mathcal{V}$ updates $\lambda_{i,k+1}= \lambda_{i,k}-\beta_k t_{i,k}$.
    \ENDFOR
    $$x_{i,N+1} = \underset{x}{\operatorname{argmin}} ~\left\{ F_i(x) + \left\langle \begin{bmatrix}
    B_i x_{i} - b_i \\
    h_i(x_{i})
\end{bmatrix},  \hat{y}_{i,N+1} \right \rangle \right\}$$
     and takes it as the final result.
\end{algorithmic}
\end{algorithm}
}

\section{Convergence Analysis}\label{sec: Convergence Analysis}
In this section, we carry out the convergence analysis for the proposed algorithm.  Specifically, we provide the convergence rates of the primal optimality error and the feasibility error.

\subsection{Convergence Rate}

\begin{theorem}\label{the: opt err}
     Consider the accelerated distributed algorithm in Algorithm  \ref{algorithm} under Assumptions \ref{ass: strongly convex} and \ref{ass: slater's condition}. Assume $\| \nabla G(y^*) \| \leq \xi$. Let $N$ be the number of iterations and the design parameters of Algorithm \ref{algorithm} be 
$
\alpha_k=\frac{2}{k+1}, \theta_k=\frac{\rho N}{k}, \beta_k=\frac{\rho k}{N}, \text { and } \eta_k=\frac{2 l_g + \rho N\|W\|}{k}.
$ 
Then, the convergence rates of the primal optimality error and the violation error of the constraints are as follows,
\begin{align}\label{violation_constraint_rate}
    \left \| \sum_{i=1}^n B_i x_{i,N+1} - b_i \right\| + \left \| \left[\sum_{i=1}^n h_i(x_{i,N+1})\right]_+ \right\| 
    \leq \varepsilon_{\rm{c}}, 
\end{align}
\begin{align}\label{primal gap}
   -\underline{\varepsilon}_{\rm{p}}\textbf{} \leq f(x_{N+1}) - f(x^*) \leq \Bar{\varepsilon}_{\rm{p}},
\end{align}
where 
\begin{align*}
    \varepsilon_{\rm{c}} =  & \ \left(\frac{2 l_g}{N(N+1)}+\frac{\rho }{(N+1)} \| W \| \right) \left\|y_1-y^*\right\|^2 \notag \\
    &\ + \frac{1 }{\rho (N+1) \lambda_2(W)},
\end{align*}
\begin{align*}
     \underline{\varepsilon}_{\rm{p}}
    =&\, \left(\frac{2 l_g}{ N(N+1)}+\frac{\rho }{(N+1)} \| W \| \right) \left\|y_1-y^*\right\|^2 \notag \\
    & + \frac{1}{\rho (N+1)}\left\| \nabla G(y^*) \right\|_{W^{\dagger}}^2 + \| y^* \| \varepsilon_{\rm{c}},
\end{align*}
and $\Bar{\varepsilon}_{\rm{p}} =\frac{1}{l_g}\left( \left(\| \nabla G(y^*) \|+l_g \| y^* \| \right) \varepsilon_{\rm{c}} + \varepsilon_{\rm{c}}^2 \right)$.
\end{theorem}
\begin{proof}
For any $\lambda \in \mathbb{R}^{n(d+m)}$, we have
\begin{align}\label{Bregman 1}
& \ G\left(\hat{y}_{k+1}\right) -G\left(y^*\right)-\left\langle \lambda, \hat{y}_{k+1}-y^* \right \rangle \notag \\
& -\left(1-\alpha_k\right)\left[G\left(\hat{y}_k\right)-G\left(y^*\right)-\left\langle \lambda, \hat{y}_{k}-y^* \right \rangle \right] \notag \\
= &\ G\left(\hat{y}_{k+1}\right)-\left(1-\alpha_k\right) G\left(\hat{y}_k\right)-\alpha_k G\left(y^*\right) \notag \\
& -\alpha_k\left\langle \lambda, y_{k+1} - y^*\right\rangle .
\end{align}
By Proposition \ref{prop:smoothness}, $G(y)$ is convex and $l_g$-smooth, and hence
\begin{align} \label{convex and smooth}
G\left(\hat{y}_{k+1}\right) \leq &\ G\left(\tilde{y}_k\right)+\left\langle\nabla G\left(\tilde{y}_k\right), \hat{y}_{k+1}-\tilde{y}_k\right\rangle \notag \\ 
& \ +\frac{ \alpha_k^2 l_g }{2 }\left\|y_{k+1}-y_k\right\|^2 \notag \\
= & \left(1-\alpha_k\right)\left[G\left(\tilde{y}_k\right)+\left\langle\nabla G\left(\tilde{y}_k\right), \hat{y}_k-\tilde{y}_k\right\rangle\right] \notag \\
& +\alpha_k\left[G\left(\tilde{y}_k\right)+\left\langle\nabla G\left(\tilde{y}_k\right), y^*-\tilde{y}_k\right\rangle\right] \notag \\
& +\!\alpha_k\left\langle\nabla G\left(\tilde{y}_k\right), y_{k+1}\!-\!y^*\right\rangle\!+\!\frac{ \alpha_k^2 l_g }{2}\left\|y_{k+1}\!-\!y_k\right\|^2 \notag\\
\leq & \left(1-\alpha_k\right) G\left(\hat{y}_k\right) +  \alpha_k G(y^*)\notag\\
& +\!\alpha_k\left\langle\nabla G\left(\tilde{y}_k\right), y_{k+1}\!-\!y^*\right\rangle \! + \! \frac{\alpha_k^2 l_g}{2}\left\|y_{k+1}\!-\!y_k\right\|^2,
\end{align}
where the second equality has been derived using \eqref{ymd} and \eqref{yag}.
Substituting \eqref{convex and smooth} into the right-hand side of \eqref{Bregman 1}, we obtain the upper bound of \eqref{Bregman 1} as
\begin{align}\label{Bregman 2}
&\ G\left(\hat{y}_{k+1}\right)-G\left(y^*\right)-\left\langle \lambda, \hat{y}_{k+1}-y^* \right \rangle \notag \\
& -\left(1-\alpha_k\right)\left[G\left(\hat{y}_k\right)-G\left(y^*\right)-\left\langle \lambda, \hat{y}_{k}-y^* \right \rangle\right] \notag \\
\leq &\ \alpha_k\left\langle\nabla G\left(\tilde{y}_k\right), y_{k+1}-y^*\right\rangle+\frac{\alpha_k^2 l_g}{2}\left\|y_{k+1}-y_k\right\|^2 \notag \\
& -\alpha_k\left\langle \lambda,  y_{k+1}-y^*\right\rangle.
\end{align}
According to \eqref{gradient ykmd}, we have the optimality condition
$$
\begin{aligned}
    0 \in & \ \nabla G \left(\tilde{y}_k\right) + \partial \mathbbm{1}_{\mathcal{Y}}( y_{k+1} ) \\
    &\ -\eta_k \left(y_k-y_{k+1}\right) - \theta_k W y_k - \lambda_k,
\end{aligned}
$$
where $\partial \mathbbm{1}_{\mathcal{Y}}( y_{k+1} )$ is a subgradient of the indicator function at $y_{k+1}$. Then, we can further express the right-hand side of \eqref{Bregman 2} as
\begin{align}\label{Bregman 3}
&\ \alpha_k\left\langle\nabla G\left(\tilde{y}_k\right), y_{k+1}-y^*\right\rangle+\frac{\alpha_k^2 l_g }{2}\left\|y_{k+1}-y_k\right\|^2 \notag \\
&-\alpha_k\left\langle \lambda, y_{k+1} - y^* \right\rangle \notag \\
= & \alpha_k\left\langle\eta_k\left(y_k-y_{k+1}\right), y_{k+1}-y^*\right\rangle+ \alpha_k\left\langle \lambda_k,  y_{k+1}-y^* \right\rangle \notag \\
& +\alpha_k \theta_k\left\langle W y_k, y_{k+1}-y^*\right\rangle-\alpha_k\left\langle  \lambda , y_{k+1} - y^* \right\rangle \notag \\
& + \langle \partial \mathbbm{1}_{\mathcal{Y}}( y_{k+1} ), y_{k+1} - y^* \rangle + \frac{\alpha_k^2 l_g}{2}\left\|y_{k+1}-y_k\right\|^2 \notag \\
\leq &\ \alpha_k\left\langle\eta_k\left(y_k-y_{k+1}\right), y_{k+1}-y^*\right\rangle \notag \\
&  + \alpha_k \left\langle  \lambda_k - \lambda, y_{k+1}-y^*\right\rangle \notag \\
& +\alpha_k \theta_k\left\langle W y_k, y_{k+1}-y^*\right\rangle+\frac{\alpha_k^2 l_g}{2 }\left\|y_{k+1}-y_k\right\|^2 \notag \\
= &\ \alpha_k\!\Big[\! \left\langle\eta_k \left(y_k -y_{k+1}\right), y_{k+1}-y^*\right\rangle \! + \! \left\langle  \lambda_{k+1}- \lambda ,  y_{k+1} - y^* \right\rangle \notag \\
& + \left\langle\left(\frac{\theta_k}{\beta_k}-1\right)\left( \lambda_k- \lambda_{k+1}\right), y_{k+1}-y^* \right\rangle \notag \\
& + \theta_k\left\langle W\left( y_{k+1} - y_k\right),  y_{k+1}-y^* \right\rangle \notag \\
& +\frac{\alpha_k l_g}{2}\left\|y_{k+1}-y_k\right\|^2\!\Big],
\end{align}
where the inequality is due to the convexity of the indicator function.
Next, by transforming inner-product terms into norms, we have 
\begin{subequations}
\begin{alignat}{3}
&\ \left\langle\eta_k\left(y_k-y_{k+1}\right), y_{k+1}-y^*\right\rangle \notag \\
= &\ \frac{\eta_k}{2}\left(\left\|y_k-y^*\right\|^2\!-\!\left\|y_{k+1}-y^*\right\|^2\!-\!\left\|y_k-y_{k+1}\right\|^2\right), \label{norms 1}\\
&\ \left\langle  \lambda_{k+1} - \lambda, y_{k+1} - y^* \right\rangle = \left\langle \lambda_{k+1}- \lambda, \frac{W^{\dagger} }{\beta_k} \left(  \lambda_k -  \lambda_{k+1} \right)\right\rangle \notag \\
= &\ \frac{1}{2 \beta_k}\left(\!\left\| \lambda_k - \lambda \right\|_{W^{\dagger}}^2 - \left\| \lambda_{k+1} - \lambda \right\|_{W^{\dagger}}^2 - \left\| \lambda_k -  \lambda_{k+1}\right\|_{W^{\dagger}}^2 \right),\label{norms 2}\\
&\ \theta_k\left\langle W (y_{k+1}-y_k), y_{k+1}-y^*\right\rangle \notag \\
= &\ \frac{\theta_k}{2}\left(\!\left\|y_k\!-\!y_{k+1}\right\|_W^2\!+\!\left\|y_{k+1}\!-\!y^*\right\|_W^2\!-\!\left\|y_k\!-\!y^*\right\|_W^2\!\right) .\label{norms 3}
\end{alignat}
\end{subequations}
By the update rule \eqref{lambda update}, we obtain that
\begin{align}\label{norms 4}
& \left\langle\left(\frac{\theta_k}{\beta_k}-1\right)\left( \lambda_k- \lambda_{k+1}\right), \left(y_{k+1}-y^*\right)\right\rangle \notag \\
= & \left\langle\left(\frac{\theta_k}{\beta_k}-1\right)\left( \lambda_k- \lambda_{k+1}\right), W^{\dagger} \left( \lambda_k- \lambda_{k+1} \right)\right\rangle \notag \\
= & \frac{\theta_k-\beta_k}{\beta_k^2}\left\| \lambda_k- \lambda_{k+1}\right\|_{W^{\dagger}}^2.
\end{align}
Combining the relations \eqref{Bregman 2}-\eqref{norms 4} yields
\begin{align*}
&\ G\left(\hat{y}_{k+1}\right)-G\left(y^*\right)-\left\langle \lambda, \hat{y}_{k+1}-y^* \right \rangle \\
&-\left(1-\alpha_k\right)\left[G\left(\hat{y}_k\right)-G\left(y^*\right)-\left\langle \lambda, \hat{y}_{k}-y^* \right \rangle\right] \\
\leq &\ \alpha_k \Bigg[\frac{\eta_k}{2}\left(\left\|y_k-y^*\right\|^2-\left\|y_{k+1}-y^*\right\|^2-\left\|y_k-y_{k+1}\right\|^2\right) \\
& \quad+\frac{1}{2 \beta_k}\left(\!\left\| \lambda_k- \lambda \right\|_{W^{\dagger}}^2-\left\| \lambda_{k+1}- \lambda \right\|_{W^{\dagger}}^2-\left\| \lambda_k- \lambda_{k+1}\right\|_{W^{\dagger}}^2\!\right) \\
& \quad-\frac{\theta_k-\beta_k}{\beta_k^2}\left\| \lambda_k- \lambda_{k+1}\right\|_{W^{\dagger}}^2 \\
& \quad+\!\frac{\theta_k}{2}\!\left(\left\|y_k-y_{k+1}\right\|_W^2+\left\|y_{k+1}-y^*\right\|_W^2-\left\|y_k-y^*\right\|_W^2\right) \\
&\left.\quad+\frac{\alpha_k l_g }{2}\left\|y_{k+1}-y_k\right\|^2\right].    
\end{align*}
Since $\alpha_k = \frac{2}{k+1}$, by multiplying $k(k+1)$ on both sides of the above inequality, we have 
\begin{align}\label{23}
& {k(k+1)} \left[ G\left(\hat{y}_{k+1}\right)-G\left(y^*\right)-\left\langle \lambda, \hat{y}_{k+1}-y^* \right \rangle\right] \notag \\
&-k(k-1)\left[G\left(\hat{y}_k\right)-G\left(y^*\right)-\left\langle \lambda, \hat{y}_{k}-y^* \right \rangle\right] \notag \\
\leq & \ 2k {\Big[\frac{\eta_k}{2}\left(\left\|y_k-y^*\right\|^2-\left\|y_{k+1}-y^*\right\|^2-\left\|y_k-y_{k+1}\right\|^2\right)} \notag \\
& +\frac{1}{2 \beta_k}\left(\left\| \lambda_k- \lambda \right\|_{W^{\dagger}}^2-\left\| \lambda_{k+1}- \lambda \right\|_{W^{\dagger}}^2-\left\| \lambda_k- \lambda_{k+1}\right\|_{W^{\dagger}}^2\right) \notag \\
& +\frac{\theta_k}{2}\left(\left\|y_k-y_{k+1}\right\|_W^2+\left\|y_{k+1}-y^*\right\|_W^2-\left\|y_k-y^*\right\|_W^2\right) \notag \\
& +\frac{\alpha_k l_g }{2}\left\|y_{k+1}-y_k\right\|^2 -\frac{\theta_k-\beta_k}{\beta_k^2}\left\| \lambda_k- \lambda_{k+1}\right\|_{W^{\dagger}}^2 \Big] .
\end{align} 
Summing both sides of \eqref{23} from $k=1$ to $k=N$ leads to
\begin{align}
&\ {N(N+1)} \left[ G\left(\hat{y}_{N+1}\right)-G\left(y^*\right) -\left\langle \lambda, \hat{y}_{N+1}-y^* \right\rangle \right] \notag\\
\leq & \left(2 l_g + \rho N \| W \| \right) \left\|y_1-y^*\right\|^2 \notag \\
&+ \frac{N}{\rho} \left(\left\| \lambda_1 -  \lambda \right\|_{W^{\dagger}}^2 -\left\| \lambda_N - \lambda \right\|_{W^{\dagger}}^2 \right). \notag
\end{align}
Considering that $\lambda_1 = 0_{(d+m)p}$ and dividing both sides of the above inequality by $N(N+1)$, we have
\begin{align}\label{GGlambda}
    &\ G\left(\hat{y}_{N+1}\right)-G\left(y^*\right)-\left\langle \lambda, \hat{y}_{N+1}-y^* \right\rangle \notag \\
    \leq & \left(\frac{2 l_g}{ N(N+1)}+\frac{\rho }{(N+1)} \| W \| \right) \left\|y_1-y^*\right\|^2 \notag \\
    & + \frac{1}{\rho (N+1)}\left\| \lambda \right\|_{W^{\dagger}}^2.
\end{align}  
If we take $\lambda = \frac{ \xi(\hat{y}_{N+1}-y^*)}{\| \hat{y}_{N+1}-y^* \|}$, the above inequality becomes
\begin{align}\label{GGnormyy}
    &\ G\left(\hat{y}_{N+1}\right)-G\left(y^*\right)+ \xi \| \hat{y}_{N+1}-y^* \| \notag \\
    \leq & \left(\frac{2 l_g}{ N(N+1)}+\frac{\rho }{(N+1)} \| W \| \right) \left\|y_1-y^*\right\|^2 \notag \\
    & + \frac{\xi}{\rho (N+1) \| \hat{y}_{N+1}-y^* \|^2} \| \hat{y}_{N+1}-y^* \|_{W^{\dagger}}^2 \notag \\
    \leq & \left(\frac{2 l_g}{ N(N+1)}+\frac{\rho }{(N+1)} \| W \| \right) \left\|y_1-y^*\right\|^2 \notag \\
    & + \frac{ \xi }{\rho (N+1) \lambda_2(W)}.
\end{align} 
By the convexity of $G$, 
\begin{align}\label{GG geq -nabla yy}
    &\ G\left(\hat{y}_{N+1}\right)-G\left(y^*\right) \geq -\| \nabla G(y^*) \| \| \hat{y}_{N+1}-y^* \|.
\end{align}
Plugging \eqref{GG geq -nabla yy} into \eqref{GGnormyy} yields
\begin{align}\label{sup_upperbound_normy}
    &\ \| \hat{y}_{N+1}-y^* \| \notag \\
    \leq & \frac{1}{\xi - \| \nabla G(y^*) \| }\Bigg[\left(\frac{2 l_g}{ N(N+1)}+\frac{\rho }{(N+1)} \| W \| \right) \left\|y_1-y^*\right\|^2 \notag \\
    & + \frac{ \xi }{\rho (N+1) \lambda_2(W)} \Bigg].
\end{align} 
Therefore, we get the upper bound for the violation of constraints 
$$
\begin{aligned}
    & \left \| \sum_{i=1}^n B_i x_{i,N+1} - b_i \right\| + \left \| \left[\sum_{i=1}^n h_i(x_{i,N+1})\right]_+ \right\| \\
    \leq & \| \nabla G(\hat{y}_{N+1}) - \nabla G(y^*) \| \notag \\
    \leq & l_g \| \hat{y}_{N+1} - y^* \| \notag \\
    \leq & \frac{1}{\xi - \| \nabla G(y^*) \| }\Bigg[\left(\frac{2 l_g^2}{ N(N+1)}+\frac{\rho l_g}{(N+1)} \| W \| \right) \left\|y_1-y^*\right\|^2 \notag \\
    & + \frac{l_g \xi}{\rho (N+1) \lambda_2(W)}\Bigg] .
\end{aligned}
$$
As for the optimality error of the primal problem, we have
\begin{align}\label{upper bound}
    &\ F(x_{N+1}) - F(x^*) \notag \\
    = & \, -G(\hat{y}_{N+1}) + G(y^*) + \langle \nabla G(\hat{y}_{N+1}), \hat{y}_{N+1}\rangle \notag \\
    = & \, -G(\hat{y}_{N+1}) + G(y^*) + \langle \nabla G(y^*), \hat{y}_{N+1}-y^* \rangle  \notag \\
    & \, + \langle \nabla G(\hat{y}_{N+1}) - \nabla G(y^*), y^*\rangle\notag \\
    &\, + \langle \nabla G(\hat{y}_{N+1}) - \nabla G(y^*), \hat{y}_{N+1} - y^*\rangle \notag \\
    \leq & \ l_g \| y^* \| \|\hat{y}_{N+1} - y^*\| + l_g \| \hat{y}_{N+1} - y^* \|^2 ,
\end{align}
where the last inequality is due to the convexity and Lipschitz smoothness of $G$. Similarly, 
\begin{align}\label{lower bound}
    &\, -F(x_{N+1}) + F(x^*) \notag \\
    = & \ G(\hat{y}_{N+1}) - G(y^*) - \langle \nabla G(\hat{y}_{N+1}), \hat{y}_{N+1}\rangle \notag \\
    = & \ G(\hat{y}_{N+1}) - G(y^*) - \langle \nabla G(y^*), \hat{y}_{N+1}-y^* \rangle  \notag \\
    & \, - \langle \nabla G(\hat{y}_{N+1}) - \nabla G(y^*), y^*\rangle\notag \\
    &\, - \langle \nabla G(\hat{y}_{N+1}) - \nabla G(y^*), \hat{y}_{N+1} - y^*\rangle \notag \\
    \leq & \, \left(\frac{2 l_g}{ N(N+1)}+\frac{\rho }{(N+1)} \| W \| \right) \left\|y_1-y^*\right\|^2 \notag \\
    & + \frac{1}{\rho (N+1)}\left\| \nabla G(y^*) \right\|_{W^{\dagger}}^2 + l_g \| y^* \| \|\hat{y}_{N+1} - y^*\|,
\end{align}
where the last inequality is by taking $\lambda = \nabla G(y^*)$ in \eqref{GGlambda}.
Inequalities \eqref{lower bound}, \eqref{upper bound} and \eqref{sup_upperbound_normy}, together with the fact that $x_{N+1}, x^* \in X$, imply \eqref{primal gap}.
\end{proof}


\section{Numerical Experiments} \label{sec: Simulation}
We consider the following distributed convex optimization problem with nonsmooth objective functions and nonlinear inequality constraints:
$$
\begin{aligned}
\underset{x_i \in X_i }{\min} & \sum_{i=1}^n x_i^{\rm{T}} A_i x_i+b_i^{\rm{T}} x_i + \| x_i \|_1 \\
\operatorname{s.t.} & \sum_{i=1}^n C_i x_i  = 0_p\\
& \sum_{i=1}^n \left\|x_i-r_i\right\|_1 \leq \sum_{i=1}^n d_i ,
\end{aligned}
$$
where $x_i \in \mathbb{R}^p$ is the local decision variable of agent $i$, and $A_i \in \mathbb{R}^{p \times p},~ b_i,~r_i \in \mathbb{R}^p,~ C_i \in \mathbb{R}^{p \times p}$, and $d_i>0$ are given problem data.

In our experimental setup, we set the number of agents to $n=20$, and the local dimension to $p=5$. Each local variable $x_i \in \mathbb{R}^p$ is subject to a box constraint $x_i \in X_i:=\left\{x_i \in \mathbb{R}^p | \underline{x}_i \leq x_i \leq \bar{x}_i\right\}$, where the lower and upper bounds $\underline{x}_i$ and $\bar{x}_i$ are independently generated with their entries drawn i.i.d. from the uniform distributions $\mathcal{U}[-10,-9]$ and $\mathcal{U}[9,10]$, respectively. Each matrix $A_i$ is constructed as $A_i=U_i \Lambda_i U_i^{\rm{T}}$, where $U_i$ is a random orthogonal matrix generated from the QR decomposition of a Gaussian matrix, and $\Lambda_i$ is a diagonal matrix with eigenvalues linearly spaced in the interval $[1, \kappa]$ with $\kappa=100$, ensuring strong convexity and a controlled condition number. The linear term coefficients $C_i \in \mathbb{R}^p$ and reference vectors $ b_i \in \mathbb{R}^p$ are generated independently from the multivariate standard normal distribution, i.e., $C_i,~ b_i \sim \mathcal{N}\left(0, I_p\right)$. Each matrix $C_i \in \mathbb{R}^{p \times p}$ is of full rank, different from the identity matrix, and randomly generated to ensure nontrivial global coupling in the equality constraint. The scalar thresholds $d_i \in \mathbb{R}_+$ are independently sampled from the uniform distribution $\mathcal{U}(1,6)$. The optimal solution is solved by using YALMIP with MOSEK, which serves as a benchmark to evaluate the accuracy and convergence performance of distributed optimization algorithms. The communication network for the simulation is constructed as a connected undirected graph $\mathcal{G}=(\mathcal{V}, \mathcal{E})$, where $\mathcal{V}=\{1,2, \ldots, 20\}$ represents the set of nodes (corresponding to buses) and the edge set $\mathcal{E}$ is defined by the rule $(i, i+1)$ for all $1 \leq i \leq 19$ and includes the edge $(1,20)$. This means each node is connected to its nearest neighbor and the next-nearest neighbor, forming a sparse yet connected topology.

To evaluate the performance of the proposed algorithm, we have compared with the state-of-the-art algorithms, the distributed subgradient algorithm \cite{liang2019distributed2} the augmented Lagrangian tracking algorithm (ALT) \cite{falsone2023augmented} and the integrated primal-dual proximal algorithm (IPLUX) \cite{wu2022distributed}. All algorithms are initialized with the same primal iterates. 
\begin{figure}[!htb]
    \centering     
    \includegraphics[width=0.45\textwidth]{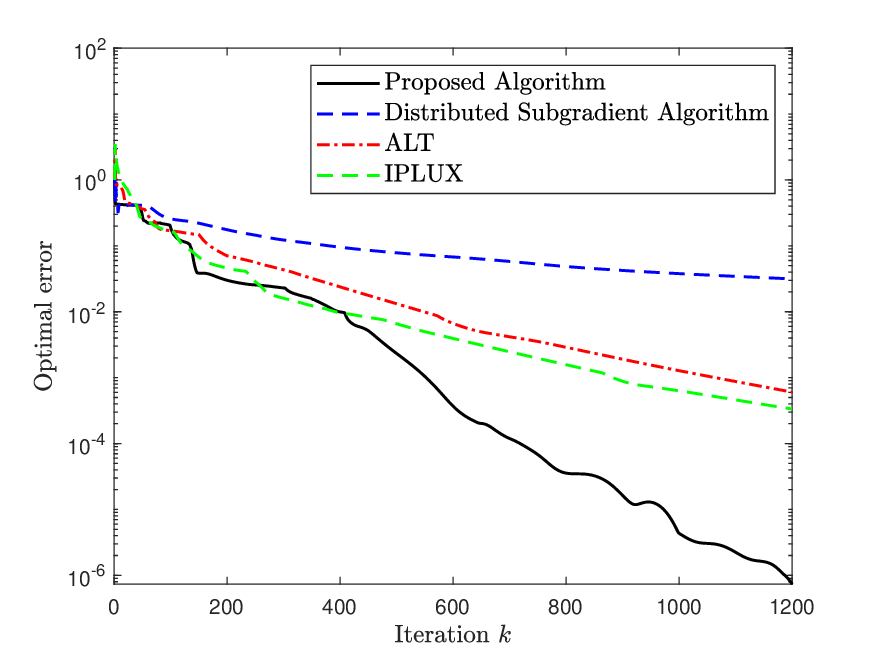}
    \caption{The primal optimal error under the four algorithms.}
    \label{fig: power mismatch}
\end{figure}

\begin{figure}[!htb]
    \centering        
    \includegraphics[width=0.45\textwidth]{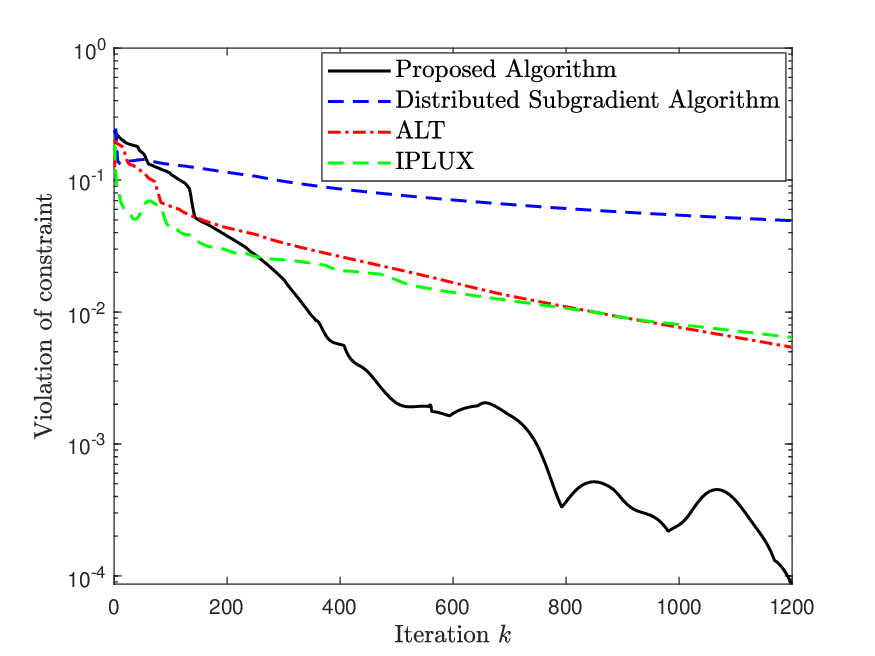}
    \caption{The absolute error of the violation of constraints under the four algorithms.}
    \label{fig: demand}
\end{figure}

Figure \ref{fig: power mismatch} plots the primal optimality error, $\frac{|f(x_{k}) - f(x^*)|^2}{|f(x_1) - f(x^*) |^2}$, under the four algorithms. The optimality error of the proposed method reaches $10^{-6}$ by $k=1200$.
In comparison, ALT decreases monotonically but more slowly, ending near $10^{-2}$, and the distributed subgradient baseline exhibits the smallest decay and remains in the $10^{-1}\!-\!10^{0}$ range over the horizon. 
Figure \ref{fig: demand} presents the absolute error of the violation of constraints, i.e., $\| \sum_{i=1}^n C_i x_{i,k} \| + [\sum_{i=1}^n (\left\|x_{i,k}-r_i\right\|_1 - d_i)]_+$. As seen, our algorithm again outperforms the others, reaching a violation error below $10^{-4}$ earlier than all other algorithms.  
Across both metrics, the proposed algorithm converges faster and has a markedly higher accuracy than the two baselines, highlighting its advantage in both speed and final precision under the same iterations.

\section{Conclusions} \label{sec: Conclusion}
We presented an accelerated distributed algorithm for the nonsmooth CCP with affine equality and nonlinear inequality couplings. By reformulating the dual as a consensus objective and combining a look-ahead linearization with a penalty on the Laplacian constraint, the algorithm keeps the per-iteration work simple, one local subproblem per node and one round of neighbor exchanges, while improving the iteration-wise rate. The analysis delivers non-ergodic bounds on both the primal optimality gap and the feasibility residual, closing a gap with prior methods that are either asymptotic or ergodic bounds. On representative testbeds, the proposed algorithm reaches tighter feasibility and lower error within the same iteration budget than augmented-Lagrangian tracking and dual subgradient algorithms. 
\bibliographystyle{IEEEtran}
\bibliography{reference}

\end{document}